\documentclass[12pt]{article}%
\usepackage{amssymb}
\usepackage{amsfonts}
\usepackage{amsmath}
\usepackage{bbm}
\usepackage{tikz}
\usetikzlibrary{positioning}
\usepackage{graphicx}%
\providecommand{\U}[1]{\protect\rule{.1in}{.1in}}

\newtheorem{theorem}{Theorem}[section]
\newtheorem{lemma}[theorem]{Lemma}

\newtheorem{proposition}[theorem]{Proposition}
\newcommand{\E}{{\mathbb E}}
\renewcommand{\P}{{\mathbb P}}
\newcommand{\R}{{\mathbb R}}
\newcommand{\eps}{{\varepsilon}}
\newcommand{\nlg}[1]{\ensuremath{{\rm nlg}\mathnormal{#1}}}
\newenvironment{proof}{\textbf{Proof.}}{\hfill $\Box$\\}
\begin{document}

\title{Explicit bounds for critical infection rates and expected extinction times of the contact process on finite random graphs }
\author{E. Cator and H. Don\thanks{Radboud University Nijmegen, Faculty of Science, P.O. Box 9010, 6500 GL
Nijmegen, The Netherlands; \emph{email}: E.Cator@science.ru.nl and H.Don@science.ru.nl},}

\maketitle

\begin{abstract}
	\noindent We introduce a method to prove metastability of the contact process on Erd\H{o}s-R\'enyi graphs and on configuration model graphs. The method relies on uniformly bounding the total infection rate from below, over all sets with a fixed number of  nodes. Once this bound is established, a simple comparison with a well chosen birth-and-death process will show the exponential growth of the extinction time. Our paper complements recent results on the metastability of the contact process: under a certain minimal edge density condition, we give explicit lower bounds on the infection rate needed to get metastability, and we have explicit exponentially growing lower bounds on the expected extinction time. 
\end{abstract}

\paragraph{Keywords:} Contact process, critical infection rate, extinction time, metastability.

\section{Introduction}\label{sec:intro}

The contact process as a model of epidemics on networks (introduced in \cite{harris74}) was first studied on infinite graphs like $\mathbb{Z}^d$ and the regular tree $\mathbb{T}_d$. Each node in the graph represents an individual, being either healthy or infected. In the latter case it infects each of its healthy neighbours at rate $\tau>0$ and heals at rate 1. All infections and healings are independent of the states of other vertices. 

One of the central questions for the process on infinite graphs is if the process dies out almost surely or not. See for instance the work of Holley and Liggett \cite{HolleyLiggett78}, showing existence of (and giving bounds for) a critical infection rate for the process on $\mathbb{Z}^d$. For infinite regular trees there even turn out to be \emph{two} critical values: the infection might survive with or without the root being infected infinitely often \cite{Pemantle92,Liggett96, Stacey96}. For an overview, we refer to \cite{liggett99,liggett05}.

Later finite random graphs were considered. In this case the process dies out almost surely, so the question is if this will take a `long' time or not. One usually takes a sequence of growing random graphs and looks at the extinction time as a function of $N$, the number of nodes in the graph. Long survival then means that the extinction time grows at least exponentially in $N$. In the recent paper \cite{BNNS19}, Bhamidi, Nam, Nguyen and Sly show that the process on the Erd\H{o}s-R\'enyi graph $G_{n,\sigma/n}$ for any $\sigma>1$ exhibits a phase transition in the following sense: if the infection rate is small enough, then the extinction time is of order $N^{1+o(1)}$ and if the infection rate is large enough, then the extinction time is exponential in $N$. Similar results are obtained for the configuration model. For constant degree the existence of a phase transition had been established before, see \cite{MV16}. For power law degree distributions however, there is no phase transition: any positive infection rate leads to exponential extinction times \cite{CD09}.

In the current work we aim to investigate the location of the phase transition. In the long survival phase, we also give lower bounds for the speed of exponential growth of the extinction time. The main idea in our approach is to lower bound the number of edges connecting a subset of the nodes to its complement. We search for uniform bounds, i.e. depending only on the size of the subset. If these lower bounds are sufficiently large for a range of sizes of subsets, this gives a positive drift to the number of infected individuals, and the expected extinction time will be exponential in the number of nodes. A more detailed explanation is given in section \ref{sec:coupling}.

We consider the dense and sparse regime of the Erd\H{o}s-R\'enyi graph and the contact process with arbitrary degree distribution. In all cases our main result is of the same flavour, and we expect similar results to hold for other types of random graphs. Consider the contact process with infection rate $\tau$ on a random graph model $G_N$ in which the expected degree of a random individual equals $\delta$. Both $\tau$ and $\delta$ might be functions of the number of nodes $N$. To guarantee existence of $\tau$ for which the extinction time $T_N$ will be exponential, we first of all need the graph to be sufficiently well connected. For instance, in the Erd\H{o}s-R\'enyi random graph, we need $\delta>1$, since otherwise the graph falls apart into small disjoint components. 

Our main results roughly say the following. For large $N$ and not too small $\delta$, exponential extinction times occur whenever $\tau\delta >1$. If $\tau\delta$ is large as well, then it generally should be true that the extinction time grows at least like $(\tau\delta)^N$. To be slightly more precise, for the graphs we consider, we prove existence of a function $f(\delta)=o(1)\,(\delta\rightarrow\infty)$ such that $\tau\delta>1+f(\delta)$ together with a proper `density condition' implies that
\begin{enumerate}
	\item The expected extinction time is (at least) exponential in $N$.
	\item There exists a function $g(\tau,\delta)=o(1)\,(\tau\delta\rightarrow\infty)$ such that 
\[
\liminf_{N\rightarrow\infty} \frac{\log(\mathbb{E}[T_N])}{N}\geq\log(\tau\delta) +g(\tau,\delta).
\]
\end{enumerate}
In section \ref{sec:ErdosRenyi} we derive these results for the Erd\H{o}s-R\'enyi graph where the density condition is that the average degree is at least $4\log(2)$. In fact this is a stronger condition than is needed for existence of a long survival phase, see \cite{BNNS19}. We leave it as an open question to extend our results to the full range $\delta>1$. In section \ref{sec:confmod}, we consider the configuration model with degree distribution $D$. In this case the density condition is $\mathbb{E}[2^{-\frac12 D}]<\frac12$. The lower bound on $\tau$ is somewhat implicit and depends on the degree distribution, and the second statement above only makes sense after specifying the degree distribution. For $D$ constant and $D\sim\text{Pois}(\mu)$, we demonstrate how to obtain an explicit lower bound on the expected extinction time.

\section{Coupling with a birth-death process}\label{sec:coupling}

The random graphs we will consider are not necessarily connected, which means that the contact process on the graph might be reducible and the contact process may die out quickly in some subset of nodes. For instance, in a sparse Erd\H{o}s-R\'enyi graph with constant expected degree, there will be a positive fraction of the nodes outside the largest connected component. If the infection vanishes in one of the components, it will never reappear there. However, if the infected fraction is not too close to 0 or 1 and the graph is not too sparse, then transitions to larger infected sets will be possible. The main point of our argument will be to control the probabilities that the infected set increases or decreases by looking at the number of links from infected nodes to healthy ones. In this analysis, we just consider the size of the infected set, rather than keeping track of which set exactly is infected. 

The contact process will live on a graph $G=(V,E)$. Let $A$ be the adjacency matrix of $G$, so $A_{ij} = 1$ if $\left\{i,j\right\}\in E$, and otherwise $A_{ij}=0$. An important characteristic of a subset of the nodes is the number of edges to its complement. This will be denoted by 
\[
L_S := \sum_{i\in S}\sum_{j\in S^c} A_{ij},\qquad S\subseteq V.
\]
Given the graph $G$, we can define the contact process $(I_t)_{t\geq 0}$. The state space is $\mathcal{S} = \mathcal{P}(V)$, the powerset of $V$. We define $\mathcal{S}_k:=\left\{S\subseteq V:|S|=k\right\}$, so that $\mathcal{S}=\bigcup_{k=0}^{|V|} \mathcal{S}_k$. We choose an infection rate $\tau$ and specify the transition rates:
\begin{equation}
\begin{array}{lccc}
{\rm For} S\subseteq V {\rm and } j\in S^c:\qquad & S\rightarrow S\cup\left\{j\right\} & \text{with rate} & \tau\sum_{i\in S} A_{ij},\\
\text{For $S\subseteq V$ and $j\in S$:}\qquad & S\rightarrow S\setminus\left\{j\right\} & \text{with rate} & 1.
\end{array}
\end{equation}
There is one absorbing state, namely the set $\emptyset$. We will be interested in the hitting time of this set, defined as $T = \inf\left\{t\geq 0:I_t=\emptyset\right\}$, which is also called the extinction time. If the graph is sufficiently well connected and $\tau$ is large enough, the process will exhibit almost stationary behaviour, and we say that there exists a metastable distribution. To properly define this notion, one usually lets the number of nodes in the graph increase. Existence of a metastable distribution is then reformulated to the fact that the expectation of $T$ grows exponentially in the number of nodes.

We will couple the process to a simple birth-death process $(P_t)_{t\geq 0}$ on a range $k_0,\ldots,k_1$. Assume this process to have strictly positive transition rates as in the diagram below.
\begin{center}
	\begin{tikzpicture}[-latex  ,node distance =3 cm and 3cm ,on grid ,
	semithick ,
	state/.style ={ circle ,top color =white , bottom color = white!20 ,
		draw, black , text=black , minimum width =1.25 cm}]
	\node[state] (1) {$k_0$};
	\node[state] (2) [right =of 1] {$k$};
	\node[state] (3) [right =of 2] {$k+1$};
	\node[state] (4) [right  =of 3] {$k_1$};
	\path (1) edge [dotted, bend left =20] node[below] {} (2);
	\path (2) edge [dotted, bend left =20] node[below] {} (1);
	\path (2) edge [bend left =20] node[above] {$\lambda_k$} (3);
	\path (3) edge [bend left =20] node[below] {$k+1$} (2);
	\path (3) edge [dotted, bend left =20] node[below] {} (4);
	\path (4) edge [dotted, bend left =20] node[below] {} (3);
	\end{tikzpicture}
\end{center}  

\begin{lemma}\label{lem:coupling}
Let $G=(V,E)$ be a graph and let $k_0,k_1\in\mathbb{N}$ be such that $0<k_0<k_1\leq |V|$. Suppose for $k_0\leq k<k_1$, there exists $M_k\in\mathbb{N}$ such that $L_S\geq M_{k}\geq 1$ for all $S\in\mathcal{S}_k$.

 Take $\lambda_k = \tau M_k$. Let $I_0=V$ and $P_0 = k_1$. Then there exists a coupling between $(I_t)_{t\geq 0}$ and $(P_t)_{t\geq 0}$ such that $|I_t|\geq P_t$ for all $0\leq t\leq \inf\left\{s\geq 0 : P_s = k_0\right\}$.
\end{lemma}

\begin{proof} We let $P_t$ and $I_t$ develop independently of each other, until $|I_t|=P_t$; suppose at that time $P_t=|I_t|=k\in \{k_0+1,\ldots,k_1\}$. Now draw independently $Z_1,Z_2\sim {\rm Exp}(1)$. If $Z_1/k < Z_2/(\tau L_{I_t})$, we randomly remove one infected node from $I_t$ to get $I_{t+Z_1/k}$. Furthermore, $P_{t+Z_1/k}=k-1$. If however $Z_1/k > Z_2/(\tau L_{I_t})$, then we choose a node from $I_t^c$ with probability proportional to the infection rate from $I_t$ to that node, and get $I_{t+Z_2/(\tau L_{I_t})}$ by adding this node. Also, we take $P_{t+Z_2/(\tau L_{I_t})}=P_t+1$ with probability $M_k/L_{I_t}$, and otherwise $P_{t+Z_2/(\tau L_{I_t})}=P_t$. Then we proceed with both processes. Since $\tau L_{I_t}\geq \lambda_k$, $P_t$ follows the correct distribution, and with our (quite natural) coupling we see that $P_t\leq |I_t|$, for all $t\leq \inf\left\{s\geq 0 : P_s = k_0\right\}$.
\end{proof}

The following lemma bounds the expected hitting time of $k_0$ in the birth-death process.

\begin{lemma}\label{lem:hitting} Let $T^{k_0}:=\inf\left\{t\geq 0:P_t = k_0\right\}$ and let $H_{k} = \mathbb{E}[T^{k_0}\mid P_0=k]$ for $k_0\leq k\leq k_1$. Then for all $k_0 < k\leq k_1$,
\[
\frac1 k_1\prod_{i=k_0+1}^{k_1-1}\frac{\lambda_i}{i} \leq H_{k} \leq \frac12 (k_1-k_0)^2 \max_{k_0\leq k\leq j\leq k_1-1} \prod_{i=k+1}^{j} \frac{\lambda_i}{i}.
\]
\end{lemma}

\begin{proof} The expected hitting times satisfy the following linear system, see e.g. \cite{norris}:
	\begin{align}\label{eq:hitlinsyst}
		H_k = \left\{\begin{array}{ll}
			0 & k = k_0,\\
			\frac{1}{k+\lambda_k}(1+k H_{k-1}+\lambda_k H_{k+1})\qquad &k_0<k<k_1,\\
			\frac{1}{k}+H_{k-1},&k=k_1.
		\end{array}
		\right.
	\end{align}
	For $k_0\leq k\leq k_1$, write $H_k = a_k H_{k_0+1}+b_k$, with initial conditions $a_{k_0}=b_{k_0}=b_{k_0+1}=0$ and $a_{k_0+1}=1$. Rewriting the second recursive relation in (\ref{eq:hitlinsyst}) then gives
	\begin{align*}
		H_{k+1} &= \frac{k+\lambda_k}{\lambda_k}\left(a_kH_{k_0+1}+b_k-\frac{k}{k+\lambda_k}(a_{k-1}H_{k_0+1}+b_{k-1})-\frac{1}{k+\lambda_k}\right)\\
		&= \left(a_k+\frac{k}{\lambda_k}\left(a_k-a_{k-1}\right)\right)H_{k_0+1}+\left(b_k+\frac{k}{\lambda_k}(b_k-b_{k-1})-\frac{1}{\lambda_k}\right)\\
		&= a_{k+1} H_{k_0+1}+b_{k+1}.
	\end{align*}
	for $k_0< k< k_1$. So we find
	\begin{align*}
		a_{k+1}-a_k &= (a_{k_0+1}-a_{k_0})\prod_{i=k_0+1}^{k}\frac{i}{\lambda_i}=\prod_{i=k_0+1}^{k}\frac{i}{\lambda_i},\\
		b_{k+1}-b_k &= (b_{k_0+1}-b_{k_0})\prod_{i=k_0+1}^k\frac{i}{\lambda_i}-\sum_{j=k_0+1}^k\frac{1}{\lambda_j}\prod_{i=j+1}^k\frac{i}{\lambda_i} = -\sum_{j=k_0+1}^k\frac{1}{\lambda_j}\prod_{i=j+1}^k\frac{i}{\lambda_i}. 
	\end{align*}
	Finally, the third equation in (\ref{eq:hitlinsyst}) gives
	\begin{align*}
		\frac{1}{k_1}&= (a_{k_1}-a_{k_1-1}) H_{k_0+1} +(b_{k_1}-b_{k_1-1}),
	\end{align*}
	so that
	\[
	H_{k_0+1} = \left[\frac{1}{k_1}+\sum_{j=k_0+1}^{k_1-1}\frac{1}{\lambda_j}\prod_{i=j+1}^{k_1-1}\frac{i}{\lambda_i}\right]\times \prod_{i=k_0+1}^{k_1-1}\frac{\lambda_i}{i}\geq \frac{1}{k_1}\prod_{i=k_0+1}^{k_1-1}\frac{\lambda_i}{i}.
	\]
	Noting that $H_k\geq H_{k_0+1}$ for all $k_0<k\leq k_1$ completes the proof of the lower bound.
		
Next, we give an upper bound for the hitting times. First observe that
\begin{align*}
H_{k_1} &= a_{k_1} H_{k_0+1}+b_{k_1} = \sum_{k=k_0}^{k_1-1}(a_{k+1}-a_k) H_{k_0+1}+\sum_{k=k_0}^{k_1-1}(b_{k+1}-b_k).
\end{align*}
The first term can be written as
\begin{align*}
\sum_{k=k_0}^{k_1-1}(a_{k+1}-a_k) H_{k_0+1} 
&
= \left(\sum_{k=k_0}^{k_1-1} \prod_{i=k+1}^{k_1-1} \frac{\lambda_i}{i}\right)\times\left(\frac{1}{k_1}+\sum_{j=k_0+1}^{k_1-1}\frac{1}{\lambda_j}\prod_{i=j+1}^{k_1-1}\frac{i}{\lambda_i}\right)\\
&  = \frac{1}{k_1}\sum_{k=k_0}^{k_1-1} \prod_{i=k+1}^{k_1-1} \frac{\lambda_i}{i} + \sum_{k=k_0}^{k_1-1}\left(-(b_{k+1}-b_k)+\sum_{j=k+1}^{k_1-1}\frac{1}{\lambda_j}\prod_{i=k+1}^j\frac{\lambda_i}{i}\right)
\end{align*}
so that
\begin{align*}
H_{k_1}&=\frac{1}{k_1}\sum_{k=k_0}^{k_1-1} \prod_{i=k+1}^{k_1-1} \frac{\lambda_i}{i} + \sum_{k_0\leq k<j\leq k_1-1}\frac{1}{\lambda_j}\prod_{i=k+1}^j\frac{\lambda_i}{i}\\
&\leq \frac{1}{k_1}\sum_{k=k_0}^{k_1-1} \prod_{i=k+1}^{k_1-1} \frac{\lambda_i}{i} + \sum_{k_0\leq k<j\leq k_1-1}\prod_{i=k+1}^{j-1}\frac{\lambda_i}{i}\\
&\leq \frac{k_1-k_0}{k_1}\max_{k_0\leq k\leq k_1-1} \prod_{i=k+1}^{k_1-1} \frac{\lambda_i}{i} + \binom{k_1-k_0}{2}\max_{k_0\leq k<j\leq k_1-1} \prod_{i=k+1}^{j-1} \frac{\lambda_i}{i}\\
&\leq \frac12 (k_1-k_0)^2 \max_{k_0\leq k\leq j\leq k_1-1} \prod_{i=k+1}^{j} \frac{\lambda_i}{i}.
\end{align*}
The conclusion follows by observing that $H_k\leq H_{k_1}$ for all $k_0<k\leq k_1$.

\end{proof}

The coupling between the two processes reveals that uniform lower bounds on numbers of edges between sets of nodes suffice to get a lower bound on the extinction time of the contact process. The next proposition summarizes this result and will be applied on two different graph models in the next sections.

\begin{proposition}\label{prop:uniform_bound}
Let $G=(V,E)$ be a graph and let $k_0,k_1\in\mathbb{N}$ be such that $0<k_0<k_1\leq |V|$. Suppose for $k_0\leq k<k_1$, there exists $M_k\in\mathbb{N}$ such that $L_S\geq M_{k}\geq 1$ for all $S\in\mathcal{S}_k$. Then the extinction time $T$ of the contact process $(I_t)_{t\geq 0}$ on $G$ with $I_0=V$ satisfies
\[
\mathbb{E}[T] \geq \frac{1}{k_1}\prod_{k=k_0+1}^{k_1-1}\frac{\tau M_k}{k}.
\]
\end{proposition}
\begin{proof}
Couple the process $(I_t)_{t\geq 0}$ with the birth-death process $(P_t)_{t\geq 0}$ with $P_0=k_1$ as in Lemma \ref{lem:coupling}. Let $T^{k_0}=\inf\left\{t\geq 0:P_t=k_0\right\}$. Then 
\[
T = \inf\left\{t\geq 0: I_t = \emptyset \right\}\geq \inf\left\{t\geq 0: |I_t| = k_0 \right\}\geq \inf\left\{t\geq 0: P_t = k_0 \right\} = T^{k_0}.
\]	
Taking expectations, Lemma \ref{lem:hitting} gives the lower bound for $\mathbb{E}[T]$.
\end{proof}

\section{Metastability for the contact process on the Erd\H{o}s-R\'enyi graph}\label{sec:ErdosRenyi}

In this section we will derive sufficient conditions for long survival of the contact process on the Erd\H{o}s-R\'enyi graph model $G_{N,p}$. We will be interested in the limit for $N$ tending to infinity and we allow the edge probability $p$ to be a function of $N$. Also the infection rate $\tau$ might be a function of $N$. We will consider the supercritical regime in which the graph has a giant component containing a positive fraction of the nodes. This regime will be split into a dense case in which the average degree $Np$ goes to infinity and a sparse case in which $Np = \sigma$ for some constant $\sigma$.

For $Np$ exceeding $4\log(2)$, we give a lower bound on the infection rate $\tau$ which is sufficient for long survival. For large degrees, this bound on $\tau$ is close to $1/(Np)$. If $ Np\tau$ is large as well, the extinction times grow almost like $(Np\tau)^N$. However, if $Np\tau$ is close to $1$, there will be correction terms in the growth. We will also show that the same correction terms occur when considering the contact process on the complete graph.

\subsection{The dense case: $Np\rightarrow \infty$}\label{sec:dense}

We will consider the contact process on a sequence of random graphs, for which we write $G_N = (V_N,E_N)$. When we start the contact process, we fix the randomly chosen graph. As discussed in Section \ref{sec:coupling}, the first goal is to find a uniform lower bound on the number of edges between sets $S\subseteq V_N$ of size $k$ and their complements. For given $S$, we denote the number of links between $S$ and $S^c$ by $L_S$. As the graph is random and depends on $N$, we will aim for a lower bound $M_{N,k}$ that is valid with probability tending to $1$ as $N$ to infinity. 

Since the condition $Np\rightarrow\infty$ is not sufficient to have a connected graph, there will be strict subsets of $V_N$ that do not have any links to their complement. However, these problematic sets are either very small or very large. We therefore choose a constant $\gamma \in (0,\frac12)$, and we will only consider sets of size $k$ with $\gamma\leq \frac k N\leq 1-\gamma$. In the next lemma we give uniform lower bounds on the number of outgoing links of such sets that hold with high probability, i.e. tending to 1 as $N\to\infty$. In the dense regime, the constant $\gamma$ can be taken as small as we wish.

\begin{lemma}\label{lem:lowerbounddense}
Consider the Erd\H{o}s-R\'enyi random graph sequence $G_N:=G_{N,p}= (V_N,E_N)$ with edge probability $p = p(N)$ and $\lim_{N\rightarrow\infty} Np = \infty$. Let $\gamma \in (0,\frac12)$ and $\rho\in (0,1)$. Then, with high probability,
\[
L_S\geq M_{N,k}:= \rho pk(N-k)
\]
for all $k$ satisfying $\gamma N\leq k\leq (1-\gamma)N$ and all $S\in\mathcal{S}_k$.
\end{lemma}

\begin{proof}
Fix $N$, $\gamma\in (0,\frac{1}2)$ and $k$ such that $\gamma N\leq k\leq (1-\gamma)N$. Also fix $S\in\mathcal{S}_k$ and consider the edges in the Erd\H{o}s-R\'enyi graph to be random. Clearly, the number of links between $S$ and $S^c$ has a binomial distribution:
\[
L_S \sim {\rm Bin}(k(N-k),p).
\]
Now use Chernoff's bound to obtain
\[
\mathbb{P}\left(L_S\leq \rho \mathbb{E}[L_S]\right)\leq e^{-\frac12(1-\rho)^2\mathbb{E}[L_S]}=e^{-\frac12(1-\rho)^2pk(N-k)} \leq e^{-\frac12(1-\rho)^2\gamma(1-\gamma)pN^2}.
\]
Bounding the binomial coefficient by $2^N$, we conclude that
\begin{align*}
\mathbb{P}(\exists S\in\mathcal{S}_k:L_S\leq \rho\mathbb{E}[L_S])&\leq \binom{N}{k}\mathbb{P}(L_S\leq \rho\mathbb{E}[L_S])\\
&\leq e^{\left(\log(2)-\frac12(1-\rho)^2\gamma(1-\gamma)pN\right)\cdot N} 
\end{align*}
Finally, we sum over $k$ and take $M_{N,k} = \rho pk(N-k)$, leading to
\[
\mathbb{P}(\exists\ \gamma N\leq k\leq (1-\gamma)N\ \exists\ S\in\mathcal{S}_k:L_S\leq M_{N,k})\leq e^{\log(N)+\left(\log(2)-\frac12(1-\rho)^2\gamma(1-\gamma)pN\right)\cdot N}. 
\]
So for $N$ large, we have with high probability that for all $\gamma N\leq k\leq (1-\gamma)N$ and all $S\in\mathcal{S}_k$ the number of links between $S$ and $S^c$ is bounded from below:
\[
L_S\geq M_{N,k} = \rho pk(N-k).
\]
\end{proof}

When we consider the graph to be random, the expected extinction time is a random variable. The next theorem gives lower bounds for this expected value that hold with high probability (so we will get a \emph{quenched} statement). It turns out that even superexponential extinction times are possible. 

\begin{theorem}\label{theorem:ERdense}
	Consider the Erd\H{o}s-R\'enyi random graph sequence $G_N:=G_{N,p}= (V_N,E_N)$ with edge probability $p = p(N)$ and $\lim_{N\rightarrow\infty} Np = \infty$. Let $(I^N_t)_{t\geq 0}$ be the contact process on $G_N$ with $I_0^N=V_N$ and infection rate $\tau = \tau(N)$. Let $T_N$ be the extinction time of the process. 
	\begin{enumerate}
		\item Let the infection rate $\tau$ be such that $N p\tau\rightarrow\infty$. Take $\varepsilon\in(0,1)$. Then with high probability
		\[
		\mathbb{E}[T_N] \geq e^{(1-\varepsilon)\log(Np\tau)N}.
		\] 
		\item Let $\lambda>1$ be a constant and let $\tau$ satisfy $Np\tau = \lambda$. Take $\varepsilon\in (0,1)$. Then with high probability
		\[
		\mathbb{E}[T_N] \geq e^{\left((1-\varepsilon)\log(\lambda)+\frac 1\lambda-1\right)\cdot N}.
		\]
	\end{enumerate}
The expectation in this theorem is taken only over the randomness of the contact process, so the graph is fixed.
\end{theorem}

\begin{proof}
To prove the first statement, we let $\tau$ be such that $Np\tau\rightarrow\infty$ and we take $0<\rho<1$ and $0<\gamma<\frac12\eps$ arbitrary. Let $k_0 = \left\lceil\gamma N\right\rceil$ and $k_1=\left\lfloor (1-\gamma)N\right\rfloor$ and apply Proposition \ref{prop:uniform_bound} using the lower bound of Lemma \ref{lem:lowerbounddense}:
\begin{align*}
\mathbb{E}[T_N] &\geq \frac{1}{k_1}\prod_{k=k_0+1}^{k_1-1}\frac{\tau M_{N,k}}{k} = \frac{1}{k_1} (\tau\rho Np)^{k_1-k_0-1}\prod_{k=k_0+1}^{k_1-1}\left(1-\frac k N\right)\\
&\geq \frac{1}{N}(\tau\gamma \rho Np)^{(1-2\gamma)N-3} \geq e^{(1-2\gamma)N\log(Np\tau)+\mathcal{O}(N)}\geq e^{(1-\varepsilon)N\log(Np\tau)},
\end{align*}
assuming $N$ is large enough. This proves the first statement.


For the second statement, suppose $Np\tau = \lambda$ and choose $1>\rho>\lambda^{-\varepsilon/3}$. Let $\gamma=1-\frac 1\lambda$ and take $k_0 = \lfloor \varepsilon\gamma N/3\rfloor-2$ and $k_1 = \lfloor \gamma N\rfloor$. Invoking Proposition \ref{prop:uniform_bound} again, we find
\begin{align*}
\mathbb{E}[T_N] &\geq \frac{1}{k_1}\prod_{k=k_0+1}^{k_1-1}\frac{\tau M_{N,k}}{k} = \frac{1}{k_1} (\rho\lambda)^{k_1-k_0-1}\prod_{k=k_0+1}^{k_1-1}\left(1-\frac k N\right).\\
\end{align*} 
Note that $k_1-k_0-1 \geq (1-\varepsilon/3)\gamma \cdot N$, so that 
\begin{align*}
\log\left((\rho\lambda)^{k_1-k_0-1}\right) & \geq \left(1-\frac \varepsilon 3\right)^2\gamma \log(\lambda)\cdot N \geq \left(1-\frac{2\varepsilon}{3}\right)\left(1-\frac 1 \lambda\right)\log(\lambda)\cdot N. 
\end{align*}
Furthermore, we have
\begin{align*}
\log\left(\prod_{k=k_0+1}^{k_1-1} \left(1-\frac k N\right)\right) &\geq \sum_{k = 0}^{k_1-1}\log\left(1-\frac k N\right)\\
& \geq N\cdot\int_{1-\gamma}^{1}\log(s)ds =  \bigl(-(1-\gamma)\log(1-\gamma)-\gamma\bigr)\cdot N\\
& = \left(\frac 1\lambda\log(\lambda)+\frac 1\lambda-1\right)\cdot N.
\end{align*}
Also, for $N$ large enough, we have 
\[
\log\left(\frac 1 k_1\right) \geq -\log(N) \geq -\frac\varepsilon 3\log\lambda \cdot N.
\]
Combining these bounds, we obtain for $N$ large enough 
\[
\mathbb{E}[T_N] \geq e^{\left((1-\varepsilon)\log(\lambda)+\frac 1\lambda-1\right)\cdot N}.
\]
\end{proof}
Since $\log(\lambda)+\frac 1\lambda-1>0$ for $\lambda>1$, the previous theorem proves that the expected extinction time of the contact process with infection rate $\tau = \lambda/Np$ grows exponentially in $N$ if $\lambda>1$.

We know that if $1/\tau$ is greater than the largest eigenvalue of the adjacency matrix $A$, the extinction time only grows logarithmically in $N$. It is not hard to see that the largest eigenvalue of $A$ is somewhere close to $Np$ (since $\sum_{i,j} \E(A_{ij}) = (N^2-N)p$), so we cannot expect that if we choose $\lambda < 1$, we would get exponential extinction time. In this sense, our method gives the optimal bound for the existence for metastability. Further research is needed to see how the contact process behaves if $1/\tau \approx Np$.

\subsection{The sparse case: $Np$ constant}\label{sec:sparse}

In this section we consider the case where $Np=\sigma>0$. It is well known that the graph consists of small (logarithmic in $N$) components in the subcritical regime where $Np<1$. For our argument to work, we need at least that all sets of size $N/2$ have links to their complements. If sets of this size can be found outside the giant component, our method will fail. So we can only hope for success if at least half of the nodes are in the giant component. This already puts a restriction on $\sigma$, since if $\sigma<2\log(2)$, the giant component will be too small. In fact, for technical reasons that will become clear soon, we have to choose $\sigma > 4\log(2)$.

The next lemma is the same in spirit as Lemma \ref{lem:lowerbounddense}, but there are some additional subtleties. First of all, we control the number of edges only for sets of size $\gamma N$, where $\gamma$ lies in a symmetric interval $(\gamma_\sigma,1-\gamma_\sigma)$ around $\frac12$. This interval becomes wider if $\sigma$ increases. For a set $S$ of size $\gamma N$, we derive a lower bound for the number of links of the form $\rho\cdot \mathbb{E}[L_S]$, but now $\rho$ will be a function of $\gamma$. This function $\rho(\gamma)$ approaches zero for $\gamma$ close to the boundaries $\gamma_\sigma$ and $1-\gamma_\sigma$. Unfortunately, explicit and optimal expressions for $\gamma_\sigma$ and $\rho(\gamma)$ seem to be out of reach. We choose them in such a way that we can provide explicit lower bounds for the expected extinction time and such that the results are asymptotically optimal for large $\sigma$. In particular this means that $\rho(\gamma)$ goes to 1 pointwise on $(0,1)$ if $\sigma$ increases. 

\begin{lemma}\label{lem:lowerboundsparse} Fix $\sigma>4\log(2)$ and consider the Erd\H{o}s-R\'enyi random graph sequence $G_N$ with edge probability $p = \frac \sigma N$. Choose
	\[
	\gamma_{\sigma} = \frac12-\sqrt{\frac14-\frac{\log(2)}{\sigma}}\in (0,\frac12)\quad\text{and}\quad\alpha_{\sigma} = \frac{2\log\left(1-2\sqrt{\frac{\log(2)}{\sigma}}\right)}{\log\left(\frac14-\frac{\log(2)}{\sigma}\right)}\in (0,2). 
	\]
	Furthermore, for $\gamma_\sigma < \gamma < 1-\gamma_\sigma$, let
	\[
	\rho(\gamma) = \left(\gamma(1-\gamma)-\frac{\log(2)}{\sigma}\right)^{\alpha_\sigma}\in (0,1)\quad\text{and}\quad M_{N,k} = \rho(k/N)\sigma k(N-k).
	\]
Then, with high probability,
\[
L_S\geq M_{N,k}
\]
for all $k$ satisfying $\gamma_\sigma N< k< (1-\gamma_\sigma)N$ and all $S\in\mathcal{S}_k$.
\end{lemma}

\begin{proof}
Fix $N$ and $k = \gamma N\in \mathbb{N}$ such that $\gamma_\sigma < \gamma < 1-\gamma_\sigma$. As before, for fixed $S\in\mathcal{S}_k$, the number of links $L_S$ to the complement has a $\text{Bin}(k(N-k),p)$ distribution. This time we apply the Chernoff-Hoeffding inequality \cite{Hoeffding} and obtain for any $\rho\in (0,1)$
\begin{align*}
\mathbb{P}(L_S\leq \rho\cdot\mathbb{E}[L_S]) &\leq e^{-k(N-k)\cdot D(\rho p || p)}\\
&= e^{-\left(\rho p\log(\rho) + (1-\rho p)\log\left(\frac{1-\rho p}{1-p}\right)\right)\gamma(1-\gamma)N^2}, 
\end{align*}
in which $D(\rho p || p)$ is the Kullback-Leibler divergence. Filling in that $p=\sigma/N$, we can find some constant $C>1$ such that for $N$ large enough,
\[ \P (L_S \leq \rho \cdot\mathbb{E}[L_I]) \leq C e^{-\sigma G(\rho)\gamma(1-\gamma) N},\]
where $G(\rho) = \rho\log(\rho)+1-\rho$. We also know that
\begin{equation}\label{eq:bin_bound}
{N\choose \gamma N} \leq e^{NH(\gamma)}, 
\end{equation}
with the entropy function $H(\gamma)$ defined by
\[ H(\gamma) = -\gamma\log(\gamma) - (1-\gamma)\log(1-\gamma).\]
Therefore,
\begin{align}\label{eq:exponent}
\mathbb{P}(\exists S\in S_{N,k}:L_S\leq \rho\cdot\mathbb{E}[L_S]) &\leq {N\choose \gamma N} \P (L_I \leq \rho \cdot\mathbb{E}[L_I])\nonumber\\
& \leq Ce^{N(H(\gamma)-\sigma G(\rho)\gamma(1-\gamma))}
\end{align}
The exponent is negative if $\sigma G(\rho) \gamma(1-\gamma) > H(\gamma)$. These functions both have their maximum at $\gamma=\frac12$, see plot below. The function $G(\rho)$ is decreasing with maximum $G(0)=1$, so the exponent can only be negative if $\sigma > 4\log(2)$ and if $\gamma_L<\gamma<\gamma_R$, see Figure \ref{fig:plotH}. 

\begin{figure}
\begin{center}
	\includegraphics[scale=0.5]{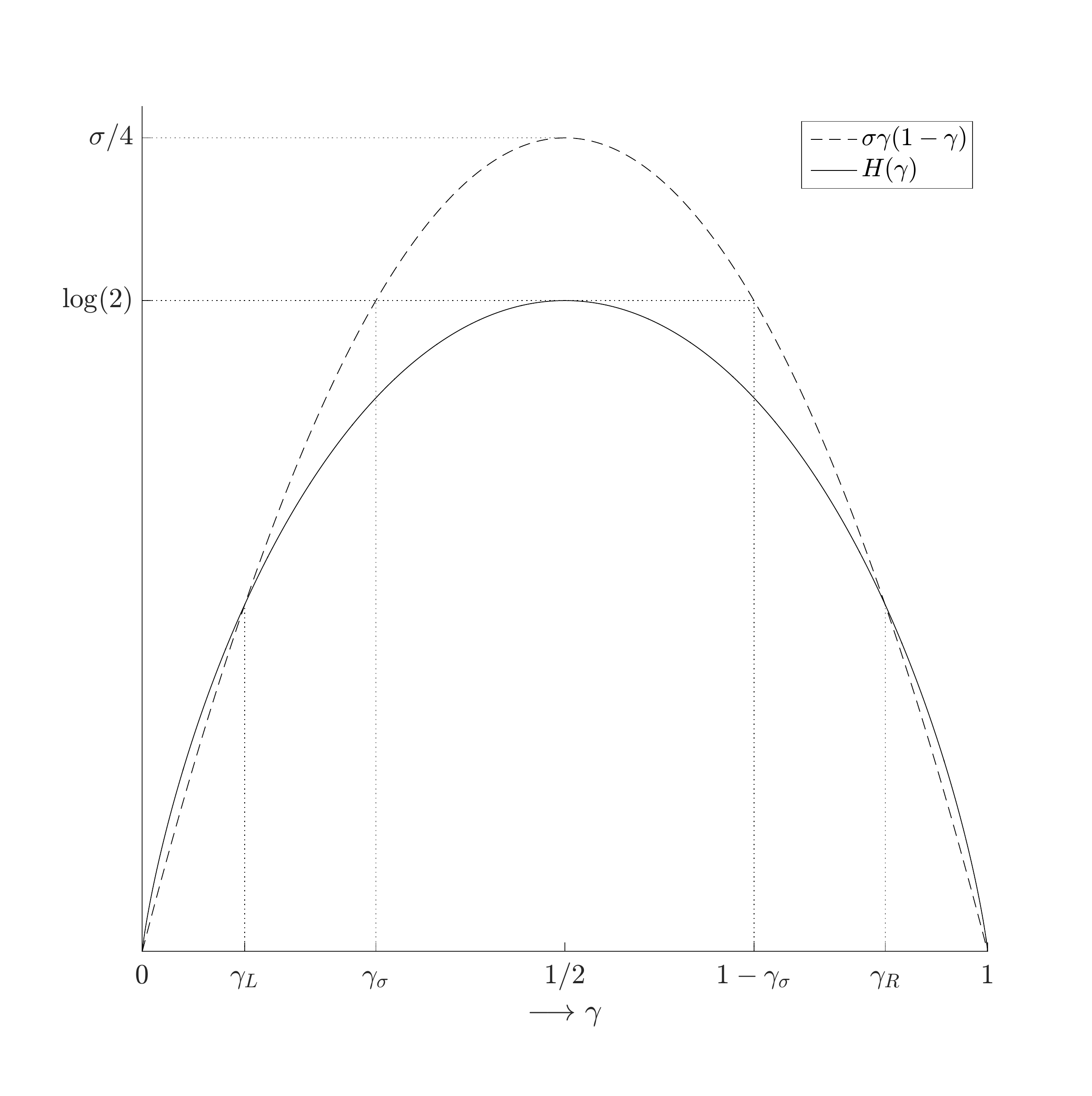}\caption{If $\sigma\gamma(1-\gamma)>H(\gamma)$, we can choose $\rho(\gamma)$ such that the exponent in (\ref{eq:exponent}) is negative.}\label{fig:plotH}
\end{center}
\end{figure}
To prove that the given choice for $\rho$ indeed guarantees the exponent to be negative, we will use that for $0<\rho<1$
\[
G(\rho) > (1-\sqrt{\rho})^2,
\]
where the right hand side is decreasing in $\rho$ as well. Moreover, note that 
\[
\alpha_\sigma = \frac{2\log\left(1-\sqrt{\frac{\log(2)}{\sigma/4}}\right)}{\log\left(\frac14-\frac{\log(2)}{\sigma}\right)} \geq \frac{2\log\left(1-\sqrt{\frac{H(\gamma)}{\sigma\gamma(1-\gamma)}}\right)}{\log\left(\gamma(1-\gamma)-\frac{\log(2)}{\sigma}\right)}
\]
for $\gamma_\sigma <\gamma<1-\gamma_\sigma$. This means the following inequalities hold
\begin{align*}
\sigma G(\rho)\gamma(1-\gamma) &> \sigma \left(1-\sqrt{\left(\gamma(1-\gamma)-\frac{\log(2)}{\sigma}\right)^{\alpha_{\sigma}}}\right)^2\gamma(1-\gamma)\\
& = \sigma\left(1-\left(1-\sqrt{\frac{H(\gamma)}{\sigma\gamma(1-\gamma)}}\right)\right)^2\gamma(1-\gamma) = H(\gamma).
\end{align*} 
It follows that 
\[
m:=\sup_{\gamma_\sigma < \gamma < (1-\gamma_\sigma)} H(\gamma) - \sigma G(\rho)\gamma(1-\gamma) <0,
\]
which implies that
\[
\lim_{N\rightarrow\infty}\mathbb{P}(\exists k\in (\gamma_\sigma N,(1-\gamma_\sigma)N), \exists S\in \mathcal{S}_{N,k}:L_S\leq M_{N,k}) \leq \lim_{N\to\infty} CNe^{Nm} = 0.
\]
\end{proof}

Having a uniform lower bound on the number of links, we proceed to bound the expected extinction time. Choose $\gamma_\sigma$ and $\alpha_\sigma$ as in Lemma \ref{lem:lowerboundsparse}. The following theorem shows that for every $\sigma>4\log(2)$ and infection rate large enough, the expected extinction time is exponential in $N$. Moreover, an explicit lower bound for the growth rate is given.

The idea of the proof essentially is that we bound the extinction time as in Proposition \ref{prop:uniform_bound}, so that our lower bound is a product of the form $\prod_\gamma \tau\rho\sigma(1-\gamma)$, where we can still choose over which interval the product is taken. We will choose an interval $(\gamma_0,\gamma_1)$ that is contained in $(\gamma_\sigma,1-\gamma_\sigma)$ and in which all terms in the product are greater than 1. To simplify notation, we let $\zeta = \frac14-\frac{\log(2)}{\sigma} = \rho(\frac12)^{\alpha_\sigma^{-1}}$. 
%
\begin{theorem}\label{thm:sparse}
		Consider the Erd\H{o}s-R\'enyi random graph sequence $G_N:=G_{N,p}= (V_N,E_N)$ with edge probability $p = \sigma/N$ for some constant $\sigma>4\log(2)$. Let $(I^N_t)_{t\geq 0}$ be the contact process on $G_N$ with $I_0^N=V_N$. There exist functions
		\[
		\tau_0(\sigma) = \frac{1+o(1)}{\sigma}\qquad\text{and}\qquad \varepsilon(\sigma)=o(1),\qquad\text{(asymptotics for $\sigma\rightarrow\infty$)}
		\]
		such that for each $\tau>\tau_0(\sigma)$ there exists $\eta>0$ for which $\mathbb{E}[T_N]>e^{\eta N}$ w.h.p. Moreover, if $\tau>\tau_0(\sigma)$,		
			\[
			\frac{1}{N}\log(\mathbb{E}[T_N]) \geq (1-\varepsilon(\sigma))\log(\tau\sigma)+\frac{1-\varepsilon(\sigma)}{\tau\sigma}-1,\qquad w.h.p.
			\]
(Explicit expressions for this lower bound and for $\tau_0(\sigma)$ are given in the proof.)
\end{theorem}
\begin{proof}
By Proposition \ref{prop:uniform_bound} and Lemma \ref{lem:lowerboundsparse}, for all $k_0$ and $k_1$ such that $\gamma_\sigma N<k_0<k_1<(1-\gamma_\sigma)N$ we have with high probability
\begin{align}\label{eq:products}
\mathbb{E}[T_N] &\geq \frac{1}{k_1}\prod_{k=k_0+1}^{k_1-1} \frac{\tau M_{N,k}}{k} =\frac{1}{k_1}(\tau\sigma)^{k_1-k_0-1} \prod_{k=k_0+1}^{k_1-1} \rho\left(\frac k N\right)\left(1-\frac{k}{N}\right).
\end{align}
To obtain a lower bound, we will choose $k_0$ and $k_1$ such that $\tau\sigma\rho(\gamma)(1-\gamma) > 1$ for all $\gamma\in(k_0/N,k_1/N)$. First we choose a constant $0<c< \max \rho(\gamma) =\rho(1/2)=\zeta^{\alpha_\sigma}$ depending on $\sigma$ and determine when $\rho(\gamma)>c$.  The equation $\rho(\gamma)=c$ has two real solutions
\[
\gamma_0 := \frac12 -\sqrt{\zeta-c^{\alpha_\sigma^{-1}}}\qquad\text{and}\qquad 1-\gamma_0.
\]
In order to obtain asymptotically optimal results, we want $c\to 1$ and $\gamma_0\to 0$ for $\sigma\to\infty$. Note that $0<\alpha_\sigma<2$ for all $\sigma>4\log(2)$. We therefore choose
\[
c = \rho(1/2)^{(\alpha_\sigma/2)^{-1/2}} = \zeta^{\sqrt{2\alpha_\sigma}}, \text{\ so that\ }\gamma_0 = \frac12 -\sqrt{\zeta-\zeta^{\sqrt{2/\alpha_\sigma}}}. 
\]
Then we use that for $\gamma\in (\gamma_0,1-\gamma_0)\cap [0,1-1/(\tau\sigma c)]$ we have
\[
\tau\sigma\rho(\gamma)(1-\gamma) > \tau\sigma c(1-\gamma)\geq 1.
\]
To have non-empty intersection, we need $\gamma_0<1-1/(\tau\sigma c)$, giving the following condition on $\tau$
\begin{equation}\label{eq:taubound}
\tau >\tau_0(\sigma) := \frac{1}{\sigma c (1-\gamma_0)} = \left(\sigma\zeta^{\sqrt{2\alpha_\sigma}}\left(\frac12 +\sqrt{\zeta-\zeta^{\sqrt{2/\alpha_\sigma}}}\right)\right)^{-1}.
\end{equation}

Now we fix $\tau$ satisfying this condition. Let $\gamma_1 = \min(1-\gamma_0,1-1/(\tau\sigma c))$ and take $k_0 = \left\lceil\gamma_0 N\right\rceil$ and $k_1 = \left\lfloor\gamma_1 N\right\rfloor$. Note that all terms in the product (\ref{eq:products}) are at least $1$ and that half of them are bounded from below by $\tau\sigma c(1-\frac{\gamma_0+\gamma_1}{2}))>1$, proving exponential growth of the extinction time.

To further bound the products in (\ref{eq:products}), we take logarithms, divide by $N$ and use that $\rho(\gamma) = ((\gamma-\gamma_\sigma)(1-\gamma-\gamma_\sigma))^{\alpha_\sigma}$:
\begin{align*}
\frac1N\log\left(\prod_{k=k_0+1}^{k_1-1} \rho\left(\frac k N\right)\right)& = \frac1N\sum_{k = k_0+1}^{k_1-1} \log\left(\rho\left(\frac k N\right)\right) \geq \int_{\gamma_0}^{\gamma_1} \log(\rho(\gamma)) d\gamma\\
&= \alpha_\sigma\int_{\gamma_0}^{\gamma_1} \log(\gamma-\gamma_\sigma)+\log(1-\gamma-\gamma_\sigma)d\gamma\\
&=\alpha_\sigma\int_{\gamma_0-\gamma_\sigma}^{\gamma_1-\gamma_\sigma}\log(s)ds +\alpha_\sigma\int_{1-\gamma_1-\gamma_\sigma}^{1-\gamma_0-\gamma_\sigma}\log(s)ds\\
&\geq -2\alpha_\sigma = -o(1)\qquad (\sigma\rightarrow\infty).
\end{align*}
Furthermore, we have
\begin{align*}
\frac1N\log\left(\prod_{k=k_0+1}^{k_1-1} \left(1-\frac k N\right)\right) &= \frac1N\sum_{k = k_0+1}^{k_1-1}\log\left(1-\frac k N\right) \geq \int_{1-\gamma_1}^1\log(s)ds\\ &= \left\{\begin{array}{ll}-\gamma_1-o(1)&\text{if}\ \gamma_1=1-\gamma_0,\\(1-\gamma_1)\log(\tau\sigma)-1-o(1)+\frac{1+o(1)}{\tau\sigma}&\text{if}\ \gamma_1 = 1-\frac{1}{\tau\sigma c}.\end{array}\right.
\end{align*}
Finally, note that for $\epsilon$ arbitrary and $N$ large enough
\begin{align*}
\frac1N\log\left(\frac{1}{k_1}(\tau\sigma)^{k_1-k_0-1}\right) &\geq \left(\gamma_1-\gamma_0-\frac 3 N\right) \log(\tau\sigma)-\frac{\log(N)}{N}\\
&\geq (\gamma_1-\gamma_0)\log(\tau\sigma)-\frac{\epsilon}{\sigma}.
\end{align*}
Putting things together, in all cases we find
\[
\frac{1}{N}\log(\mathbb{E}[T_N])\geq (1-o(1))\log(\tau\sigma)-1+\frac{1+o(1)}{\tau\sigma}-o(1)\qquad (\sigma\rightarrow\infty),
\]
which implies the statement of the theorem.
\end{proof}

\begin{figure}
	\begin{center}
		\includegraphics[scale=0.5,angle = 270]{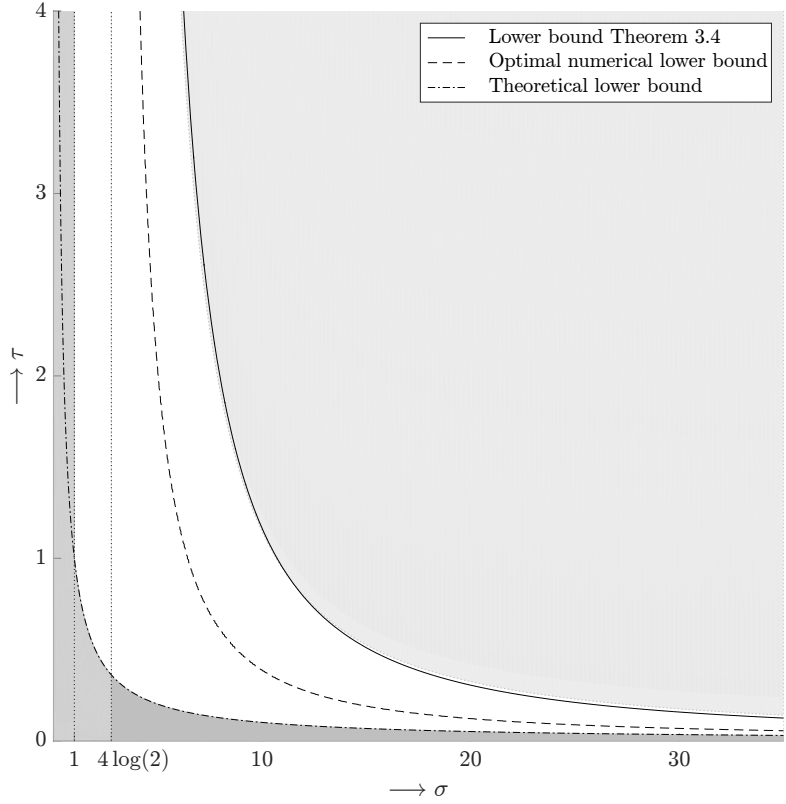}\caption{Comparison of different lower bounds for the infection rate as functions of $\sigma$. Expected extinction times are subexponential in $N$ in the darker shaded region. In the lighter shaded region, Theorem \ref{thm:sparse} gives explicit exponential lower bounds.}\label{fig:minimal_tau}
	\end{center}
\end{figure}

Thus, for $\sigma> 4\log(2)$, we found $\tau_0$ such the extinction time grows exponentially in $N$ whenever $\tau>\tau_0$. This lower bound is not optimal, see Figure \ref{fig:minimal_tau} for a plot of $\tau_0$ as function of $\sigma$, compared with a numerical approximation of the optimum that could be achieved by our current method of bounding links.  We have $\zeta\to \frac14$ and $\alpha_\sigma\to 0$ for large $\sigma$, so that $\tau_0$ approaches $\frac 1 \sigma$ and therefore is asymptotically optimal. Below this threshold, which is plotted as well, the expected extinction times cannot be exponential. Also for $\sigma<1$, extinction times are subexponential, since the random graph only consists of small components. Rigorous results in the white area require further research, it would especially be interesting to know what happens in the vertical strip $\sigma\in [1,4\log(2)]$. In \cite{BNNS19}, it is shown that for $\sigma$ in this range, there exists an infection rate making the expected extinction time exponential, but our methods are not powerful enough to quantify this. 

For large $\sigma$, both the exponential growth speed of $T_N$ and the critical infection rate coincide with the results for the dense case. We conclude this section with the observation that also the contact process on the complete graph gives results that are consistent with our findings for the Erd\H{o}s-R\'enyi graph. 

The contact process on the complete graph $K_N$ is a pure birth-death process on the state space $\left\{0,\ldots,N\right\}$. Taking infection rate $\tau = \lambda/N$ for some $\lambda>1$, the birth rate in state $k$ is $\lambda_k = \lambda k (N-k)$, while the death rate is $k$. By Lemma \ref{lem:hitting}, the expectation of the extinction time $T_N$ can be bounded. Taking $k_0=0$ and $k_1 = N$, we find the upper bound
\begin{equation}\label{eq:completegraph}
\mathbb{E}[T_N]\leq \frac{N^2}{2} \prod_{k=0}^{\lfloor(1-\frac 1 \lambda)N\rfloor} \frac{\lambda (N-k)}{N}.
\end{equation}
For the lower bound, we use Proposition \ref{prop:uniform_bound} with $k_0=1$, $k_1 = \lceil(1-\frac 1 \lambda)N\rceil$ and $M_k = k(N-k)$. This gives 
\[
\mathbb{E}[T_N]\geq \frac{1}{\lceil(1-\frac 1 \lambda)N\rceil} \prod_{k=2}^{\lceil(1-\frac 1 \lambda)N\rceil-1} \frac{\lambda (N-k)}{N}.
\]
Similar calculations as before show that for $N\rightarrow\infty$ and $\lambda>1$ the expected extinction time grows exponentially and
\[
\lim_{N\rightarrow\infty}\frac{1}{N}\log(\mathbb{E}[T_N]) = \log(\lambda)+\frac 1\lambda-1>0.
\]
This result for the contact process on the complete graph displays exactly the growth speed that we found in our lower bounds on the Erd\H{o}s-R\'enyi graph. To find upper bounds for the Erd\H{o}s-Renyi graph a bigger effort is needed, since our Lemma \ref{lem:hitting} requires to bound the number of outgoing links for sets of all sizes. 

Note that the bound on the infection rate for the complete graph is sharp as well, since for $\lambda\leq 1$ all terms in the product in (\ref{eq:completegraph}) are less than 1 so that only the polynomial factor remains. Therefore, the expected extinction time is exponential in $N$ if and only if $\lambda >1$.

\section{Configuration model}\label{sec:confmod}

In this section we will consider the configuration model. We consider a random variable $D\in \{0,1,2,\ldots\}$ and an i.i.d. sequence of degrees $D_1, \ldots, D_N$ with $D_1\sim D$. Since $D_i=0$ means that this node will never interact with other nodes, we could leave them out and still have $\Theta(N)$ nodes left. However, we will not assume $D\geq 1$ since our methods work without this assumption as well. 

Given the degree sequence, the configuration model creates a random graph by completely randomly assigning the ``half-stubs'' to each other, creating links in that way. We will not care about self-loops, double links or an odd number of total degrees: self-loops and a left-over stub are ignored, and multiple links will simply mean that the infection rate between such nodes is an integer multiple of $\tau$. 

We will derive bounds for the number of links between a set of nodes and its complement. We will consider an arbitrary $S\in \mathcal{S}_{k}$, where $k=\lceil\gamma N\rceil$. In that case, we expect the number of links between $S$ and $S^c$ to be linear in $N$, if we keep a fixed degree distribution. Choose $l=\rho N$. Our goal is to find a combination of $\gamma$ and $\rho$ such that with sufficiently high probability the sets $S$ and $S^c$ have at least $\rho N$ links between them. Define the random variable $L_S$ as the number of outgoing links of set $S$. First we will study the number of stubs in $S$ and $S^c$.

Define the independent random variables
\[
S_1= \sum_{i\in S} D_i \mbox{ and }  S_2 =  \sum_{j\in S^c} D_j.
\]
We need the following trivial consequence of Cram\'er's Theorem. Define
\[
R(x) = \sup_{\lambda\in\mathbb{R}}(\lambda x-\log(\mathbb{E}(e^{\lambda D})))
\]
to be the rate function of the random variable $D$. The function $R(x)$ is non-negative, convex and satisfies $R(\mathbb{E}[D])=0$. 
\begin{lemma}\label{lemma:ldp_stubs}
	For closed intervals $F_1,F_2\subset \R$, we have that
	\[
	\limsup_{N\rightarrow\infty}\frac1N\log(\mathbb{P}(S_1/N\in F_1))\leq -\inf_{x\in F_1} \gamma R(\frac x\gamma)
	\]
	and
	\[
	\limsup_{N\rightarrow\infty}\frac1N\log(\mathbb{P}(S_2/N \in F_2))\leq -\inf_{x\in F_2} (1-\gamma) R(\frac x{1-\gamma}).
	\]
\end{lemma}
\begin{proof}
	Cram\'er's Theorem gives for any closed set $F\subseteq \mathbb{R}$,
	\[
	\limsup_{N\to\infty} \frac{1}{\lceil\gamma N\rceil}\log(\mathbb{P}(\frac{S_1}{\lceil\gamma N\rceil}\in F))\leq -\inf_{x\in F} R(x).
	\]
	For $\eps>0$, define $F_1^{\gamma,\eps} = \frac{1}{\gamma}\cdot F_1 + [-\eps/\gamma,\eps/\gamma]$. Then for $N$ large enough, and since $F_1$ is an interval,
	\[ \frac{N}{\lceil\gamma N\rceil}\cdot F_1 \subset F_1^{\gamma,\eps}.\]
	Therefore,
	\[ \limsup_{N\rightarrow\infty}\frac1N\log(\mathbb{P}(S_1/N\in F_1))\leq 
	-\inf_{y\in F_1^{\gamma,\eps}} \gamma R(y) = -\inf_{x\in F_1^{1,\eps}} \gamma R(\frac{x}{\gamma}).\]
	Since $R$ is a lower semi-continuous function and $F_1$ is an interval, we can take the limit for $\eps \downarrow 0$ to conclude that
	\[
	\limsup_{N\rightarrow\infty}\frac1N\log(\mathbb{P}(S_1/N\in F_1))\leq -\inf_{x\in F_1} \gamma R(\frac x\gamma).
	\]
	The second statement follows completely analogously.
\end{proof}

From the previous lemma we conclude that with high probability the numbers of stubs in $S$ and $S^c$ will not be too small. The next step in our argument is to show that this implies that the number of links $L_S$ between $S$ and $S^c$ is unlikely to be small. If we have two sets of nodes, one set having $n_1$ stubs and the other having $n_2$ stubs, then the probability distribution of the number of links $L$ between these two sets is given in the following elementary combinatorial result, given without proof.
\begin{lemma}\label{lem:comb}
	Suppose we have a vase with $n_1$ red balls and $n_2$ white balls, and we take them out pairwise, completely at random. If $B=n_1+n_2$ is odd, one ball will stay left behind. Denote by $L$ the number of mixed pairs that is drawn, i.e., pairs consisting of a red and a white ball. First suppose that $B$ is even. Then for all $l$ such that $0\leq l\leq \min(n_1,n_2)$ and such that $n_1-l$ (and therefore also $n_2-l$) is even, we have
	\[ \P(L=l) = \frac{2^l {{\frac12B}\choose {\ l,\frac12 (n_1-l), \frac12 (n_2-l)\ }}}{{{B}\choose{n_1}}} = \frac{2^l {{\frac12B}\choose{l}}{{\frac12B-l}\choose{\frac12(n_1-l)}}}{{{B}\choose{n_1}}}.\]
	For all other $l$ we have $\P(L=l)=0$. Now suppose that $B$ is odd. Take $0\leq l\leq \min(n_1,n_2)$. If $n_1-l$ is even (and therefore $n_2-l$ is odd), we have
	\[ \P(L=l) = \frac{n_2}{B}\cdot \frac{2^l {{\frac12(B-1)}\choose {\ l,\frac12 (n_1-l), \frac12 (n_2-l-1)\ }}}{{{B-1}\choose{n_1}}} = \frac{2^l {{\frac12(B-1)}\choose{l}}{{\frac12(B-1)-l}\choose{\frac12(n_1-l)}}}{{{B}\choose{n_1}}}.\]
	If $n_1-l$ is odd, we have
	\[ \P(L=l) = \frac{n_1}{B}\cdot \frac{2^l {{\frac12(B-1)}\choose {\ l,\frac12 (n_1-l-1), \frac12 (n_2-l)\ }}}{{{B-1}\choose{n_1-1}}} = \frac{2^l {{\frac12(B-1)}\choose{l}}{{\frac12(B-1)-l}\choose{\frac12(n_1-l-1)}}}{{{B}\choose{n_1}}}.\]
\end{lemma}

The probabilities given in this lemma go to zero exponentially fast for most choices of $l$. In the next lemma, we give a left tail estimate for $L_S$ given that the numbers of stubs of $S$ and $S^c$ are $n_1$ and $n_2$. Note that the expected number of links in case of exactly $n_1$ and $n_2$ stubs is approximately $\lambda:=n_1n_2/(n_1+n_2)$. We denote $x\log(x)$ by $\nlg(x)$ and define $\nlg(0)=0$. 

Define the function $\phi:\mathbb{R}^3_+\to [0,\infty)$ by
\[
\phi(a_1,a_2;\rho) = \frac12 \nlg(a_1+a_2)+\frac12\nlg(a_1-\rho)+\frac12\nlg(a_2-\rho)-\nlg(a_1)-\nlg(a_2)+\nlg(\rho)
\]
whenever $a_1a_2\geq \rho(a_1+a_2)$ and let $\phi(a_1,a_2;\rho)=0$ otherwise. Since $\phi(a_1,a_2,\rho)=0$ when $a_1a_2= \rho(a_1+a_2)$, it follows that $\phi$ is a continuous function. Note that $\phi$ is decreasing in $\rho$. Furthermore, it satisfies the following scaling property
\begin{align}\label{eq:scalephi}
\phi(a_1N,a_2N;\rho N) = \phi(a_1,a_2;\rho) N.
\end{align}

\begin{lemma}\label{lemma:links|stubs}
	Let $G=(V,E)$ be a configuration model graph with degree distribution $D$. Let $S\subseteq V$ and let $L_S$ be the number of links between $S$ and $S^c$. For all $n_1,n_2,l\in\mathbb{N}$,
	\begin{equation}\label{eq1:phi}
	\P\left(L_S\leq l\Bigm| \sum_{i\in S} D_i = n_1, \sum_{j\in S^c} D_j = n_2\right) \leq e(l+1)^{3/2} e^{-\phi(n_1,n_2;l)}.
	\end{equation}
\end{lemma} 
\begin{proof} Fix $n_1$ and $n_2$. If $n_1n_2\leq l(n_1+n_2)$, then $\phi(n_1,n_2;l)=0$ and (\ref{eq1:phi}) clearly holds. From now on assume $n_1n_2> l(n_1+n_2)$, which in particular implies that $l<\min(n_1,n_2)$. We will use the following bounds that come from Stirling's approximation:
	\[
	\frac12\log(2\pi n)+\nlg(n)-n<\log(n!)<\frac12\log(2\pi n)+\nlg(n)-n+\frac{1}{12n}.
	\]
	We will assume that $B=n_1+n_2$ even, for $B$ odd a similar approach works. By Lemma \ref{lem:comb}, we may assume that $n_1-l$ and $n_2-l$ 	are even as well. In that case
	\[
	\mathbb{P}(L_S=l)  = \frac{2^l(\frac12(n_1+n_2))!n_1!n_2!}{l!(\frac12(n_1-l))!(\frac12(n_2-l))!(n_1+n_2)!}.
	\]
	It follows by Stirling's bounds that  for $l\geq 1$ (and therefore $n_1,n_2\geq 2$)
	\begin{align*}
	\log(\mathbb{P}(L=l)) \leq &\  l\log(2) + \nlg(\frac12(n_1+n_2))+\nlg(n_1)+\nlg(n_2)-\nlg(l)\\&\ -\nlg(\frac12(n_1-l)) -\nlg(\frac12(n_2-l))-\nlg(n_1+n_2)\\
	&\ -\frac12\log(2\pi)+\frac12\log\left(\frac{\frac12(n_1+n_2)n_1n_2}{l\frac12(n_1-l)\frac12(n_2-l)(n_1+n_2)}\right)\\
	&+\frac{1}{6(n_1+n_2)}+\frac{1}{12n_1}+\frac{1}{12n_2}\\
	= &\  -\phi(n_1,n_2;l)+\frac12\log\left(\frac{n_1n_2}{\pi l(n_1-l)(n_2-l)}\right)+\frac{1}{6(n_1+n_2)}+\frac{1}{12n_1}+\frac{1}{12n_2}.\\
	\end{align*}
	Since $l(n_i-l)\geq n_i-1\geq n_i/2$, we obtain
	\[
	\log(\mathbb{P}(L_S=l))\leq -\phi(n_1,n_2;l)+\frac12\log(l)+1.
	\]
	For $l=0$, we find
	\[
	\log(\mathbb{P}(L_S=l))\leq -\phi(n_1,n_2;l)+1,
	\]
	so that 
	\[
	\P(L_S=l) \leq e\sqrt{l+1}\cdot e^{-\phi(n_1,n_2;l)}
	\]
	for all $l\geq 0$. Since $\phi$ is decreasing in $l$, we have $\phi(n_1,n_2;k)\geq \phi(n_1,n_2;l)$ for $k\leq l$, and therefore
	\begin{equation*}
	\P\left(L_S\leq l\Bigm| \sum_{i\in S} D_i = n_1, \sum_{j\in S^c} D_j= n_2\right) \leq e\sum_{k=0}^{l}\sqrt{k+1}e^{-\phi(n_1,n_2;k)} \leq e(l+1)^{3/2}e^{-\phi(n_1,n_2;l)}. 
	\end{equation*}
\end{proof}

We expect the number of links between $S$ and $S^c$ to be of order $N$, so we choose $l=\rho N$. If $\rho$ is smaller than the expected fraction of links, the probability to have less than $l$ links goes to zero exponentially fast. The next lemma quantifies the rate of convergence, and it is an adaptation of Varadhan's Integral Lemma. 

\begin{lemma}\label{lem:upperbound}
	Fix $\gamma \in (0,1)$. For any $S\subseteq V$, $|S|=k$, with $k=\lceil \gamma N\rceil$, and any $\rho> 0$, we have that
	\[\limsup_{N\to \infty} \frac{1}{N}\log\left(\P(L_S\leq \rho N)\right) \leq -\inf_{a_1,a_2\geq 0}\left[ \phi(a_{1},a_{2}; \rho) +\gamma R\left(\frac{a_1}{\gamma}\right) +(1-\gamma) R\left(\frac{a_2}{1-\gamma}\right)\right].\]
Note that for $\rho\geq \gamma(1-\gamma)\E(D)$, the right-hand side equals $0$.
\end{lemma}
\begin{proof}
	Fix a large integer $K$ and define
	\[ a_j= \frac{j}{\sqrt{K}} \mbox{ for } j=0,1,\ldots K\mbox{ and } a_{K+1}=+\infty.\]
	Furthermore, define
	\[ A_{N,j} = [a_jN, a_{j+1}N[ \mbox{ for } j=0,1,\ldots, K.\]
	Then
	\begin{align*}
	\P(L_S\leq \rho N)& = \sum_{j_1=0}^{K} \sum_{j_2=0}^{K} \P(L_S\leq \rho N \mid S_1 \in  A_{N,j_1}, S_2\in A_{N,j_2}) \P(S_1 \in  A_{N,j_1})\P(S_2\in A_{N,j_2})\\
	& \leq \sum_{j_1=0}^{K} \sum_{j_2=0}^{K} \P(L_S\leq \rho N \mid S_1 = \lceil a_{j_1}N\rceil, S_2=\lceil a_{j_2}N \rceil) \P(S_1 \in  \overline{A_{N,j_1}})\P(S_2\in \overline{A_{N,j_2}})\\
	\end{align*}
	We used the fact that $(n_1,n_2)\mapsto \P(L_S\leq \rho N \mid S_1 = n_1, S_2= n_2)$ is decreasing in $n_1$ and in $n_2$: the more stubs you have, the more likely it is to have more links. 
	
 We can use this to get a large deviation result, also using property \eqref{eq:scalephi} of $\phi$ and Lemma \ref{lemma:ldp_stubs} and Lemma \ref{lemma:links|stubs}:
	\[
	\limsup_{N\to \infty} \frac{1}{N}\log\left(\P(L_S\leq \rho N)\right)  \leq \sup_{0\leq j_1,j_2\leq K}\left[ -\phi(a_{j_1},a_{j_2}; \rho) -\inf_{a_1\in [a_{j_1},a_{j_1+1}]}\gamma R(\frac{a_1}{\gamma}) -\inf_{a_2\in [a_{j_2},a_{j_2+1}]}(1-\gamma) R(\frac{a_2}{1-\gamma})\right].
	\]
	Since $\phi$ is a well-behaved function and $R$ is convex, we can take the limit for $K\to \infty$, and conclude that
	\[
	\limsup_{N\to \infty} \frac{1}{N}\log\left(\P(L_S\leq \rho N)\right)  \leq -\inf_{a_1,a_2\geq 0}\left[ \phi(a_1,a_2; \rho) + \gamma R(\frac{a_1}{\gamma}) + (1-\gamma) R(\frac{a_2}{1-\gamma})\right].\]
The last comment follows from taking $a_1=\gamma\E(D)$ $a_2=(1-\gamma)\E(D)$. If $\rho\geq \gamma(1-\gamma)\E(D)$, then $\phi(a_1,a_2;\rho)=0$, and therefore the right-hand side equals $0$.
\end{proof}

Inspired by the previous lemma, we define for $\gamma\in(0,1)$ and $\rho\in[0,\gamma(1-\gamma)\E(D))$
\begin{equation}\label{eq:defPsi}
 \Psi(\gamma,\rho) = \inf_{(a_1,a_2)\in \R_+^2} \phi(a_1,a_2;\rho) + \gamma R(a_1/\gamma) + (1-\gamma)R(a_2/(1-\gamma)).\end{equation}
This function quantifies the rate at which the probability that a set  $S$ of size $\gamma N$ has less than $\rho N$ links to $S^c$ goes to zero. Since we aim at a uniform lower bound, this rate has to be sufficiently large to compensate for the binomial coefficient $\binom{N}{\gamma N}$, i.e. we look for $\gamma$ and $\rho$ for which $\Psi(\gamma,\rho)>H(\gamma)$. The next lemma gives a sufficient condition for existence of a uniform lower bound for sets of size around $N/2$.

\begin{lemma}\label{lem:uniform_configuration} Let $G=(V,E)$ be a configuration model graph on $N$ nodes with degree distribution $D$. If $\mathbb{E}[2^{-\frac12 D}]<\frac12$, then there exist $\gamma\in (0,\frac12)$ and $\rho\in (0,\gamma(1-\gamma)\mathbb{E}[D])$ such that $\Psi(\gamma,\rho)>H(\gamma)$. Moreover, with high probability, 
\[
L_S \geq M_{N,k} := \rho N
\]
for all $\gamma N\leq k\leq (1-\gamma) N$ and all $S\in\mathcal{S}_k$.
\end{lemma}

\begin{proof}
We will maximize $\Psi(\gamma,\rho)$. Since $\phi$ is decreasing in $\rho$, $\Psi(\gamma,\rho)$ is decreasing in $\rho$ as well. For $\rho=0$ and all $(a_1,a_2)\in\mathbb{R}^2_+$ the inequality $a_1a_2\geq \rho(a_1+a_2)$ holds, so that $\phi(a_1,a_2;0)=\frac12(a_1+a_2)\log(a_1+a_2)-\frac12a_1\log(a_1)-\frac12a_2\log(a_2)$ for $(a_1,a_2)\in\R_+^2$.
Therefore 
\begin{align}
\Psi(\gamma,0) &= \inf_{(a_1,a_2)\in \R_+^2} \phi(a_1,a_2;0) + \gamma R(a_1/\gamma) + (1-\gamma)R(a_2/(1-\gamma))\nonumber \\
& =  \inf_{(u_1,u_2)\in \R_+^2} \phi(\gamma u_1,(1-\gamma)u_2;0) + \gamma R(u_1) + (1-\gamma) R(u_2).\label{eq:defPsi0}
\end{align}
It is not hard to check that $\gamma\mapsto \phi(\gamma u_1,(1-\gamma)u_2;0)$ is a concave function on $(0,1)$ for all $(u_1,u_2)\in \R^2_+$, so $\Psi(\gamma;0)$ is a positive concave function, symmetric in $\gamma=\frac12$ and therefore maximal in $\gamma = \frac12$. At $\gamma=\frac12$, we find
\begin{align*}
\Psi(\frac12 , 0) &= \inf_{(u_1,u_2)\in \R^2_+} \frac14(u_1+u_2)\log(\frac12(u_1+u_2)) - \frac14 u_1\log(\frac12 u_1)-\frac14 u_2\log(\frac12 u_2) +\frac12(R(u_1) + R(u_2))\\
&= \inf_{u\in \R_+} \frac12u\log(u) - \frac12 u\log(\frac12 u)+R(u)\\
&= \inf_{u\in \R_+} \frac12 \log(2)u + R(u)\\
& = -\sup_{u\in \R_+} -\frac12 \log(2)u - R(u)\\
& = - \log\left(\E\left(e^{-\frac12 \log(2)D}\right)\right).
\end{align*}
For the second equality we use that on the line $u_1+u_2=c$, the functions $\phi$ and $R(u_1)+R(u_2)$ are convex and take their minimal value for $u_1=u_2$. The last equality follows from the fact that the rate function $R$ is the Legendre transform of the cumulant generating function $\lambda\mapsto \log(\E(e^{\lambda D}))$, and the Legendre transform is an involution.
This implies that
\[ \Psi(\frac12,0)>H(\frac12) \iff \E(2^{-\frac12 D}) < \frac12. \]
Since $\Psi(\gamma,\rho)$ is continuous in both arguments, there exists $\gamma_0\in (0,\frac12)$ and $\rho\in (0,\gamma(1-\gamma)\mathbb{E}[D])$ such that $\Psi(\gamma,\rho)>H(1/2)\geq H(\gamma)$ for all $\gamma\in (\gamma_0,1-\gamma_0)$ whenever $\E(2^{-\frac12 D}) < \frac12$.

By Lemma \ref{lem:upperbound}, for $N$ large enough and $\gamma_0 N \leq k\leq (1-\gamma_0) N$, 
\[
\mathbb{P}(\exists k\in[\gamma_0N,(1-\gamma_0)N],\exists S\in\mathcal{S}_{N,k}:L_S\leq \rho N)\leq N\cdot\binom{N}{N/2} e^{-N\cdot\Psi(\gamma,\rho)}\leq N\cdot e^{-N\cdot(\Psi(\gamma,\rho)-H(1/2))},
\]
which goes to zero if  $\E(2^{-\frac12 D}) < \frac12$.
\end{proof}

Using Jensen's inequality we see that
\[ \E(2^{-\frac12 D}) \geq 2^{-\frac12\E(D)},\]
so our condition implies that $\E(D)>2$, which in turn implies that $\E(D^2)>2\E(D)$; this condition implies that there is one giant component in the graph with high probability. If our condition holds, and the infection rate $\tau$ is sufficiently large, then the extinction time of the contact process will be exponential in the number of individuals. 

Our condition focuses on the case $\gamma=\frac12$. However, there might exist degree distributions for which $\Psi(\frac12,0)\leq H(\frac12)$ but $\Psi(\gamma,0)>H(\gamma)$ for some $\gamma \in (0,\frac12)$. We don't know if such distributions exist, but we do know that if $\P(D\leq 1)>0$, then
\[ \lim_{\gamma\to 0} \frac{\Psi(\gamma,0)}{H(\gamma)} = 0.\]
This is a somewhat technical result that we state without proof, and it just shows that for a fixed distribution, it does not make sense to look at really small $\gamma$.

Our next theorem gives a (somewhat implicit) lower bound on $\tau$. Define the set
\[ \Gamma = \{ \gamma\in(0,\frac12)\mid \Psi(\gamma,0)>H(\gamma)\}.\]
If $\Gamma=\emptyset$, then our method cannot be used for that particular distribution of $D$. When $\Gamma\neq \emptyset$, by continuity of $\phi$ there exist $\gamma\in \Gamma$ and $\rho>0$ for which $\Psi(\gamma,\rho)>H(\gamma)$. In particular, this is the case if $\E(2^{-\frac12 D}) < \frac12$, as we have seen in Lemma \ref{lem:uniform_configuration}. For $\Gamma\neq \emptyset$, we define 
\begin{align}\label{eq:muzero}
 \mu_0= \sup\{\ \frac{\rho}{\gamma}\, \mid \gamma\in \Gamma,\rho \in (0,\gamma(1-\gamma)\E(D)) \mbox{ and }\Psi(\gamma,\rho)>H(\gamma)\}.
\end{align}
Note that $\mu_0\leq \E(D)$.

\begin{theorem}\label{theorem:confmod}
Let $G = (V,E)$ be a configuration model graph on $N$ nodes with degree distribution $D$. If $\mathbb{E}[2^{-\frac12 D}]<\frac12$ and $\tau >1/\mu_0$, then there exists a constant $c>0$ such that with high probability the extinction time $T_N$ of the contact process $(I_t)_{t\geq 0}$ on $G$ with $I_0=V$ satisfies
\[
\mathbb{E}[T_N]>e^{cN}.
\]
\end{theorem}

\begin{proof}
By Lemma \ref{lem:uniform_configuration} and continuity of $\Psi$, there exist $\gamma_0\in (0,\frac12)$, $\gamma_1>\gamma_0$ and $\rho >0 $ such that for all $\gamma\in (\gamma_0,\gamma_1)$ we have that $\Psi(\gamma,\rho)>H(\gamma)$ and 
\[
\frac{\rho\tau}{\gamma} >\frac{\rho\tau}{\gamma_1} >1.
\]
Let $k_0 = \lceil\gamma_0 N\rceil$ and $k_1=\lfloor \gamma_1 N\rfloor$. By Proposition \ref{prop:uniform_bound},
\[
\mathbb{E}[T_N] \geq \frac{1}{k_1}\prod_{k=k_0+1}^{k_1-1}\frac{\tau\rho N}{k}\geq \frac{1}{k_1} (\frac{\tau\rho}{\gamma_1})^{k_1-k_0-1} = (\frac{\tau\rho}{\gamma_1})^{(\gamma_1-\gamma_0)N-\mathcal{O}(\log(N))},
\]
implying the result of the theorem.
\end{proof}

\subsection*{Two examples}

We will consider our method for two examples, namely constant degree (leading to the random regular graph) and the Poisson distribution. If the expected degrees are large, then the asymptotic results and correction terms in these two examples are consistent with the results for the Erd\H{o}s-R\'enyi graph.

\subsubsection*{Constant degree}

First suppose that $\P(D=d)=1$, for some $d\in \{1,2,\ldots\}$. It is not hard to see that in that case,
\begin{align*}
\Psi(\gamma,0) &= \phi(d\gamma,d(1-\gamma);0)\\
&=\frac12d\log(d)-\frac12\gamma d\log(\gamma d)-\frac12(1-\gamma)d\log((1-\gamma)d)\\
& = \frac12dH(\gamma).
\end{align*}
This means that if $d\geq 3$, we will be able to find $\tau$ large enough, such that the expected extinction time grows exponentially with $N$. The next proposition gives a lower bound on $\tau$ that is close to $d^{-1}$. For slightly larger $\tau$, we also derive a lower bound for the growth rate. For large $d$, the expectation essentially grows like $(\tau d)^N$.

\begin{proposition} Let $T_N$ be the extinction time of the contact process $(I_t)$ on a configuration model graph $G=(V,E)$ with constant degree $d\geq 3$ and $I_0=V$. 
	\begin{enumerate}
		\item If $\tau>\frac 1 {d-2}$, there exists a constant $c>0$ such that w.h.p.
		\[
		\mathbb{E}[T_N]\geq e^{cN}.
		\] 
		\item Let $\lambda_d = 1-\sqrt{\log(2)/d}$. If $\tau >\frac{1}{d\lambda_d}$ and $\varepsilon>0$, then w.h.p.
		\[
		\mathbb{E}[T_N]\geq e^{((1-\varepsilon)\log(\tau d\lambda_d)+\frac 1{\tau d \lambda_d}-1)\cdot N}
		\]
	\end{enumerate} 
\end{proposition}
\begin{proof}
To get a bound on $\tau$, we determine $\mu_0$ as defined in (\ref{eq:muzero}). We see that $\Gamma=(0,\frac12)$ and 
\begin{align*}
\Psi(\gamma,\rho) & = \frac12d\log(d) - \gamma d\log(\gamma d) - (1-\gamma)d\log((1-\gamma)d) + \frac12(\gamma d-\rho)\log(\gamma d-\rho) +\\
&\ \  \,\, +  \frac12((1-\gamma)d-\rho)\log((1-\gamma)d-\rho) + \rho\log(\rho).
\end{align*}
Define $\rho=\lambda\gamma(1-\gamma)d$, for $\lambda\in [0,1]$. Note that $\gamma(1-\gamma)d N$ is the expected number of links between a set of size $\gamma N$ and its complement. With this parametrization, $\Psi$ simplifies to
\[
\Psi(\gamma,\rho) =  \frac{d}{2}\Bigl(H(\gamma)-\gamma H(\lambda(1-\gamma))-(1-\gamma)H(\lambda\gamma)\Bigr).
\]

For $\gamma$ close to zero, the leading order term in $H(\gamma)$ is $-\gamma\log(\gamma)$ and we find
\[
\Psi(\gamma,\rho) = \frac{d}{2}(H(\gamma)-H(\lambda\gamma))+\mathcal{O}(\gamma) = \frac{d}{2}(1-\lambda)H(\gamma)+\mathcal{O}(\gamma).
\]
In order to have $\Psi(\gamma,\rho)>H(\gamma)$ in a neighborhood of zero, we need 
\begin{align}\label{eq:restriction1}
\lambda<\frac{d-2}{d}.
\end{align}
Since $\lim_{\gamma\downarrow 0} \rho/\gamma = \lambda d$, we find that $\mu_0\geq d-2$. By Theorem \ref{theorem:confmod}, if we have a configuration model graph with constant degree $3$ or higher, and $\tau>1/(d-2)$, then we will have an exponentially growing expected extinction time.

Next we aim at a more explicit lower bound for $\mathbb{E}[T_N]$. We wish to find $\lambda$ such that $\Psi(\gamma,\rho)>H(\gamma)$ for all $\gamma\in (0,1)$. For $\gamma = \frac12$, we obtain
\[
\Psi(\frac12,\rho) = \frac{d}{2}\left(\log(2)-H(\frac{\lambda}{2})\right),
\]
and the corresponding condition on $\lambda$ is
\begin{align}\label{eq:restriction2}
H(\frac \lambda 2) < (1-\frac 2 {d}) H(\frac12) = (1-\frac 2 {d})  \log(2).
\end{align}
It turns out that if (\ref{eq:restriction2}) is satisfied, then $\Psi(\gamma,\rho)>H(\gamma)$ for all $\gamma\in (0,1)$.  For each $d\geq 3$, there exists a unique maximal $\lambda_0\in (0,1)$ such that this inequality holds for $\lambda<\lambda_0$. Moreover, if the degree $d$ is large, $\lambda_0$ will be close to 1, meaning that we can choose $\rho N$ close to the expected number of links. Unfortunately, $\lambda_0$ can not be calculated explicitly, so we will use that for $x\in [0,1]$
\[
H(x)\leq 2x(1-x)-\frac12+\log(2).
\]
This gives us that $\lambda_0 > 1-\sqrt{\log(2)/d}=\lambda_{d}$, which still goes to 1 for large $d$. Summarizing, we conclude that with high probability for each $\gamma\in (0,1)$ and each set $S$ of size $\gamma N$, the number of outgoing links is at least $\lambda_{d}\gamma(1-\gamma)d N$.

By Proposition \ref{prop:uniform_bound}, for $1<k_0<k_1\leq N$,
\[
\mathbb{E}[T_N] \geq \frac 1 k_1\prod_{k=k_0+1}^{k_1-1} \tau d\lambda_d\left(1-\frac k N\right).
\]
We choose $k_0 = 1$ and $k_1 = \lfloor (1-1/(\tau d\lambda_d))N\rfloor$. The second claim of the theorem then follows analogously to the proof of Theorem \ref{theorem:ERdense}.
\end{proof}

\subsubsection*{Poisson degree distribution}

Now consider $D\sim {\rm Pois}(\mu)$. First note that
\[
\mathbb{E}[e^{\lambda D}] = e^{\mu(e^\lambda-1)},
\]
so that 
\[
\mathbb{E}[2^{-\frac12 D}] = e^{\mu(\frac{1}{\sqrt 2}-1)}<\frac12 \qquad \text{if}\qquad \mu>\frac{\log(4)}{2-\sqrt{2}}\approx 2.36.
\]
By Theorem \ref{theorem:confmod}, if $\mu$ satisfies this condition and $\tau$ is large enough, the expected extinction time will be exponential in $N$. We can even improve this lower bound on $\mu$ by only looking at non-isolated nodes in the graph. They constitute again a configuration model graph with degree distribution ${\rm Pois}(\mu)$, but now conditioned on being non-zero. This gives
\[
\mathbb{E}[2^{-\frac12 D}\mid D\geq 1] = \frac{e^{\frac{\mu}{\sqrt 2}}-1}{e^\mu-1}<\frac12 \qquad \text{if}\qquad \mu>1.88.
\] 
To simplify our calculations, we will from now on work with the unconditioned degree distribution. The rate function for $D$ is given by
\[
R(x) = \sup_{\lambda\in\mathbb{R}}(\lambda x-\log(\mathbb{E}[e^{\lambda D}])) = 
\left\{\begin{array}{ll}
\infty&x<0,\\ x\log(\frac x \mu)-x+\mu & x\geq 0. 
\end{array}\right.
\]
To find $\Psi(\gamma,\rho)$, we minimize the function
\[
\psi(a_1,a_2):= \phi(a_1,a_2;\rho)+\gamma R(a_1/\gamma)+(1-\gamma) R(a_2/(1-\gamma))
\]
as in (\ref{eq:defPsi}). It turns out that $\psi$ is convex in $a_1$ and $a_2$ and setting the partial derivatives equal to zero gives the equations
\[
\left\{\begin{array}{ll}
(a_1+a_2)(a_1-\rho) = \mu^2\gamma^2,\\
(a_1+a_2)(a_2-\rho) = \mu^2(1-\gamma)^2.
\end{array}
\right.
\]
The solutions are
\[
a_1 = \rho+\frac{\mu^2\gamma^2}{\sqrt{\rho^2+\mu^2(\gamma^2+(1-\gamma)^2)}+\rho},\qquad a_2 = \rho+\frac{\mu^2(1-\gamma)^2}{\sqrt{\rho^2+\mu^2(\gamma^2+(1-\gamma)^2)}+\rho}.
\]
Choosing $\rho = \lambda\gamma(1-\gamma)\mu$ for $\lambda\in [0,1]$, it follows after some calculations that 
\[
\Psi(\gamma,\rho) = \mu\cdot\bigl(1-s(\gamma)+\lambda\gamma(1-\gamma)\log(s(\gamma)) \bigr),
\]
where
\[
s(\gamma) :=   \lambda\gamma(1-\gamma)+\sqrt{\lambda^2\gamma^2(1-\gamma)^2+\gamma^2+(1-\gamma)^2}.
\]
Since $\Psi$ increases linearly in $\mu$, for every allowed combination of $\gamma$ and $\rho$ there exists $\mu$ large enough such that $\Psi(\gamma,\rho)>H(\gamma)$. In particular, if $\mu$ is large, we can choose $\lambda(1-\gamma)$ close to 1, so that $\mu_0$ is slightly smaller than $\mu$. By Theorem \ref{theorem:confmod}, we will have exponential expected extinction time if the infection rate $\tau$ is slightly larger than $1/\mu$. 

The next proposition gives more explicit results on the minimal infection rate and the rate of exponential growth of the extinction time. Our explicit calculations turn out to work for $\mu>\frac{8\log(2)}{2-\log(2)}\approx 4.24$.

\begin{proposition} Let $T_N$ be the extinction time of the contact process $(I_t)$ on a configuration model graph $G=(V,E)$ with $\text{Pois}(\mu)$ degree distribution and $I_0=V$. 
	\begin{enumerate}
		\item Suppose  $e^\mu-1>2e^{\frac{\mu}{\sqrt 2}}-2$. Then there exists an infection rate $\tau$ and constant $c>0$ such that w.h.p. $\mathbb{E}[T_N]>e^{cN}$.
		\item Suppose $\mu>8\log(2)/(2-\log(2))$. The functions $f(\mu)$ and $g(\tau,\mu)$, explicitly given in (\ref{eq:funcf}) and (\ref{eq:funcg}), satisfy
		\[
		f(\mu) = \frac{1+o(1)}{\mu}>\frac{1}{\mu},\qquad g(\tau,\mu) = (1-o(1))\log(\tau\mu)+\frac{1+o(1)}{\tau\mu}-1
		\]
		(asymptotics for $\mu\rightarrow\infty$) and are such that w.h.p.
		\[
		\mathbb{E}[T_N] \geq e^{g(\tau,\mu)\cdot N}
		\]
		whenever $\tau>f(\mu)$.
	\end{enumerate} 
\end{proposition}

\begin{proof} The first statement follows from the discussion above. To obtain the second result, we continue by bounding $s(\gamma)$. Straightforward calculations show that
\[
1-c_1\gamma(1-\gamma) \leq s(\gamma) \leq 1-c_2\gamma(1-\gamma),
\]
where $c_1 = 4-\lambda-\sqrt{\lambda^2+8}$ and $c_2 = 1-\lambda$. Furthermore, using that 
\[
\min_{\gamma\in [0,1]}s(\gamma) = \frac{\lambda}{4}+\frac14\sqrt{\lambda^2+8}\qquad\text{and}\qquad\log(x)\geq \frac{(1-x)\log(a)}{1-a}
\]
for $0<a\leq x\leq 1$, we obtain 
\[
\log(s(\gamma))\geq (1-s(\gamma))\cdot\frac{\log\Bigl((\lambda+\sqrt{\lambda^2+8})/4\Bigr)}{1-(\lambda+\sqrt{\lambda^2+8})/4} =: (1-s(\gamma))\cdot c_3.
\]
Combining these bounds gives (note that $c_3<0$ and $c_1c_3\geq -2\log(2)c_2$)
\begin{align*}
\Psi(\gamma,\lambda\gamma(1-\gamma)\mu)&\geq \mu\Bigl(1-(1-c_2\gamma(1-\gamma))+c_3\lambda\gamma(1-\gamma)(1-s(\gamma)\Bigr)\\
& \geq \mu\Bigl(c_2\gamma(1-\gamma)+c_1c_3\lambda\gamma^2(1-\gamma)^2\Bigr)\\
& \geq (1-\lambda)\mu\gamma(1-\gamma)\cdot\Bigl(1-2\log(2)\lambda\gamma(1-\gamma)\Bigr). 
\end{align*}
Now suppose 
\[
\lambda = 1-\frac{c_4}{\sqrt{\mu}}, \qquad \text{with}\qquad c_4 = \sqrt{\frac{8\log(2)}{2-\log(2)}}.
\]
Then for all $\mu> c_4^2$, we have $\lambda\in (0,1)$ and
\[
\Psi\left(\frac12,\frac{\lambda\mu}{4}\right) \geq \frac{2-\log(2)}{8}c_4\sqrt{\mu}>\log(2) = H\left(\frac12\right),
\]  
giving the desired inequality $\Psi(\gamma,\rho)>H(\gamma)$ for $\gamma=\frac12$. Solving the equation
\[
 (1-\lambda)\mu\gamma(1-\gamma)\cdot\Bigl(1-2\log(2)\lambda\gamma(1-\gamma)\Bigr) = 2\gamma(1-\gamma)-\frac12+\log(2)
\]
gives 
\[
\gamma(1-\gamma) = \frac{(1-\lambda)\mu-2-\sqrt{((1-\lambda)\mu-2)^2-\log(2)(8\log(2)-4)\lambda(1-\lambda)\mu}}{4\log(2)\lambda(1-\lambda)\mu},
\]
so that we obtain solutions $\gamma_0\in (0,\frac12)$ and $1-\gamma_0$ given by
\[
 \frac12\pm \sqrt{\frac{(\log(2)\lambda-1)(1-\lambda)\mu+2+\sqrt{((1-\lambda)\mu-2)^2-\log(2)(8\log(2)-4)\lambda(1-\lambda)\mu}}{4\log(2)\lambda(1-\lambda)\mu}}.
\]
For large $\mu$, the two solutions approach $0$ and $1$:
\[
\gamma_0 =  \left(\frac{\log(2)}{2}-\frac14\right)\sqrt{\frac{2-\log(2)}{2\log(2)}}\cdot\frac{1}{\sqrt{\mu}} + o\left(\frac{1}{\sqrt{\mu}}\right).
\]
We conclude that for all $\mu>c_4^2$, we can choose $\lambda\in(0,1)$ and $0<\gamma_0<\frac12$ such that $\Psi(\gamma,\lambda\gamma(1-\gamma)\mu)>H(\gamma)$ for all $\gamma\in(\gamma_0,1-\gamma_0)$. Consequently, with high probability all sets of size $\gamma N$ with $\gamma\in(\gamma_0,1-\gamma_0)$ will have at least $\lambda\gamma(1-\gamma)\mu N$ links to the complement. 

Finally, we bound the expected extinction time.  Suppose 
\begin{equation}\label{eq:funcf}
\tau > f(\mu) := \frac{1}{(1-\gamma_0)(\sqrt{\mu}-c_4)\sqrt{\mu}},
\end{equation}
and define
\[ \gamma_1 := \min\left\{1-\gamma_0,1-\frac{1}{(\sqrt{\mu}-c_4)\sqrt{\mu}\tau}\right\}\in(\gamma_0,1-\gamma_0].
\]
Choosing $k_0 = \lceil \gamma_0 N \rceil$ and $k_1 = \lfloor\gamma_1N\rfloor$ and $\varepsilon>0$ arbitrary, by Proposition \ref{prop:uniform_bound}, w.h.p.
\begin{align*}
\mathbb{E}[T_N] &\geq \frac 1 k_1 \prod_{k=k_0+1}^{k_1-1} \tau (\sqrt{\mu}-c_4)\sqrt{\mu}(1-\frac k N)\nonumber\\
& \geq \frac 1 N (\tau(\sqrt{\mu}-c_4)\sqrt{\mu})^{(\gamma_1-\gamma_0)N-3}e^{N\int_{1-\gamma_1}^{1}\log(s)ds}\nonumber\\
& = \exp\left(\left(\log(\tau(\sqrt{\mu}-c_4)\sqrt{\mu})(\gamma_1-\gamma_0)+\int_{1-\gamma_1}^{1}\log(s)ds-\frac{\log(N)-3}{N}\right)\cdot N\right)\nonumber\\
& \geq e^{g(\tau,\mu)\cdot N},
\end{align*}
where 
\begin{equation}\label{eq:funcg}
g(\tau,\mu):=\log(\tau(\sqrt{\mu}-c_4)\sqrt{\mu})(\gamma_1-\gamma_0)+\int_{1-\gamma_1}^{1}\log(s)ds-\frac{\varepsilon}{\mu}.
\end{equation}
This function is increasing in both $\mu$ and $\tau$ and for small enough $\varepsilon$ it satisfies $g(\tau,\mu)>0$ if $\tau>f(\mu)$. Furthermore, for $\mu\rightarrow\infty$\
\[
g(\tau,\mu) = (1-o(1))\log(\tau\mu)+\frac{1+o(1)}{\tau\mu}-1,
\]
completing the proof.
\end{proof}

So it turns out that for large $\mu$, we will have exponential extinction time for $\tau$ slightly larger than $1/\mu$, just like in the critical Erd\H{o}s-R\'enyi graph with large average degree. Interestingly, our method proves existence of $\tau$ giving exponential growth if the mean degree exceeds 1.88. In the Erd\H{o}s-R\'enyi case we needed a stronger condition on the average degree as it had to be greater than $4\log(2)\approx 2.77$. This illustrates the fact that more subtle methods are needed to get good results if the average degree is close to 1.

We would like to mention that the method we have shown here is not able to predict that for heavy tailed degree distributions, we will have exponential extinction time for any $\tau > 0$. Also, we do not claim that our bounds for $\tau$ are optimal: this would require further research. We already know that our conditions for the existence of $\tau>0$ with exponential extinction time are not optimal, thanks to the results in \cite{BNNS19}, but they are not able to give explicit bounds for such $\tau$.

\bibliographystyle{plain}
\bibliography{Contact_Process.bib}

\end{document}